\documentclass[12pt]{amsart}
\usepackage{amsfonts, amsbsy, amsmath, amssymb}

\usepackage[T1]{fontenc}
\usepackage[utf8]{inputenc}
\usepackage{lmodern}
\usepackage{amsmath, amsthm, amssymb,amscd, mathrsfs, amsfonts, mathtools,tikz-cd}
\usepackage{relsize}
\usepackage{euler}
\usepackage{times}
\usepackage[all]{xy}
\usepackage{todonotes}
\usepackage{xcolor}
\usepackage[bookmarks=true,
            bookmarksnumbered=true, breaklinks=true,
            pdfstartview=FitH, hyperfigures=false,
            plainpages=false, naturalnames=true,
            colorlinks=true,pagebackref=true,
            pdfpagelabels]{hyperref}

\usepackage{hyperref}

\usepackage[lite]{amsrefs}

\renewcommand{\PrintDOI}[1]{\href{http://dx.doi.org/\detokenize{#1}}{doi: \detokenize{#1}}%
  \IfEmptyBibField{pages}{, (to appear in print)}{}}

\def\commutatif{\ar@{}[rd]|{\circlearrowleft}}

\newtheorem{thm}{Theorem}[section]
\newtheorem{pro}[thm]{Proposition}
\newtheorem{lemma}[thm]{Lemma}
\newtheorem{cor}[thm]{Corollary}
\newtheorem{conj}[thm]{Conjecture}

\theoremstyle{definition}
\newtheorem{definition}[thm]{Definition}

\newtheorem{qst}[thm]{Question}

\theoremstyle{remark}
\newtheorem{remark}[thm]{Remark}

\newtheorem{exs}[thm]{Examples}
\allowdisplaybreaks

\newcommand{\Z}{\mathbb{Z}}

\usepackage{bbm}

\def\Aut{\operatorname{Aut}}
\def\Inn{\operatorname{Inn}}

\def\a{\alpha}
\def\b{\beta}
\def\g{\gamma}

\date{}

\usepackage{datetime2}
\date{\today}

\title{ Idempotents and Powers of Ideals in Quandle Rings}

\author{Valeriy Bardakov}
\address{Regional Scientific and Educational Mathematical Center of Tomsk State University, 36 Lenin Ave., Tomsk, Russia.}
\address{Sobolev Institute of Mathematics, 4 Acad. Koptyug avenue, 630090, Novosibirsk, Russia.}
\address{Novosibirsk State Agrarian University, Dobrolyubova street, 160, Novosibirsk, 630039, Russia.}
\email{bardakov@math.nsc.ru}

\author{Mohamed Elhamdadi} 
\address{Department of Mathematics and Statistics, 
University of South Florida, Tampa, FL 33620, U.S.A.} 
\email{emohamed@usf.edu} 

\begin{document}
\maketitle

\begin{abstract}

This article addresses two central problems in the theory of quandle rings. 
First, motivated by Conjecture 3.10 in Internat. J. Math. 34 (2023), no. 3, Paper No. 2350011: for a semi-latin quandle \(X\), every nonzero idempotent in the integral quandle ring \(\mathbb{Z}[X]\) necessarily corresponds to an element of \(X\), we investigate idempotents in quandle rings of semi-latin quandles.  Precisely, we prove that if the ground ring is an integral domain with unity, then the quandle ring of Core($\mathbb{Z}$) admits only trivial idempotents.  Second, powers of augmentation ideals in quandle rings have only been computed in a few cases previously.  We extend the computations to include dihedral quandles and commutative quandles.  Finally, we examine idempotents in quandle rings of \(2\)-almost latin quandles and apply these results to compute the automorphism groups of their integral quandle rings.

\end{abstract}

\tableofcontents

\section{Introduction}
Quandles are algebraic structures whose axioms are motivated by the three Reidemeister moves in knot theory.  They were introduced independently by D. Joyce \cite{Joyce} and S. Matveev \cite{Matveev} in the nineteen eighties.  Quandles have been used extensively to construct invariants of knots,  knotted surfaces and virtual knots (see for example \cite{EN, CJKLS, CES, CEGS}).  Quandles also arise in a variety of mathematical contexts, including set-theoretic solutions to the Yang-Baxter equation \cite{BES}, group theory \cite{EMR, BDS, BNS}, category theory \cite{CCES} and representation theory \cite{EM, ESZ}.  Concepts from ring theory were introduced to quandle theory in \cite{BPS2}.  Precisely, a notion of a quandle ring was introduced by analogy with a group ring  and relationships between subquandles of the given quandle and ideals of the associated quandle ring were investigated in \cite{BPS2}.  It was proved in \cite{EFT} that quandle rings of non-trivial quandles are not power-associative rings and that if the ground ring is Noetherian and the quandle is finite then the quandle ring is both left and right Noetherian.  Zero divisors in quandle rings were investigated in \cite{BPS1}.  An initiation to homological study of quandle rings was given in \cite{EMSZ} where a complete characterization of derivations of quandle algebras of dihedral quandles over fields of characteristic zero has been given, and dimension of the Lie algebra of derivations has been investigated.   The question of whether a given quandle ring possesses non-trivial idempotents was initiated in  \cite{BPS1}. In \cite{ENSS}  a complete characterization of the idempotents of certain quandle coverings was given.  In \cite{ENS} some applications of quandle rings to knot theory were introduced.  Precisely, an enhancement of the coloring quandle invariant was obtained by considering the algebra homomorphisms between the quandle ring of a link, and a fixed target quandle ring.  Furthermore, when the idempotent set of a quandle ring is a quandle itself, it was shown that the quandle coloring by the idempotent quandle gives an enhancement of the quandle coloring invariant.  The combination of idempotents in quandle rings and quandle $2$-cocycles was used in \cite{ES} to obtain an enhancement of the coloring invariant in \cite{CESY} thus distinguishing knots with up to \emph{thirteen} crossings knots using only \emph{twenty one} connected quandles.
 
 In this article, we treat quandle rings as algebraic objects in their own right rather than by their connections with knot theory as in \cite{ENS, ES}.  We consider two main problems.  The first problem is the investigation of Conjecture 3.10 in \cite{ENSS}.  Second,  we extend the computations of the powers of the augmentation ideals  to include dihedral quandles and commutative quandles.  
 
 The remainder of this article is organized as follows.  In Section~\ref{review} we review the basics of quandles and quandle rings.  Section~\ref{infinitequandles} deals with infinite quandles.  In particular, we prove that if the ground ring is unital without zero-divisors then the quandle ring of the Core quandle of infinite cyclic group  has only trivial idempotents (cf. Theorem~\ref{Thm3.1}). Remark  that the Core quandle $Core(\Z)$ is semi-latin but not latin. On the other hand, this quandle is left orderable but is not right orderable.  In Section~\ref{Idemp-Dihedral} we compute the powers of the augmentation ideal for dihedral quandles.  Section~\ref{PowerIdealCommutative} gives a complete classification of the powers of the augmentation ideal over the commutative quandles of odd prime orders $5$ and $7$ (cf. Proposition~\ref{CommutativeZ5} and Proposition~\ref{CommutativeZ7}).  In Section~\ref{2-almost} we initiate the study of idempotents over $m$-almost latin quandles.  As an example, for the involutory connected quandle $X$ of order \emph{six},  we completely characterize the idempotents over its integral quandle ring $\mathbb{Z}[X]$ (cf. Proposition~\ref{IdX}).  As a consequence, we obtain the automorphisms of the quandle ring $\mathbb{Z}[X]$.  Section~\ref{Questions:} gives some open questions for further research.

\section{Basics of Quandles and Quandle Rings}\label{review}
We review the basics of quandles needed for this article. All definitions and known results can be found in the book \cite{BES}.

\begin{definition}
A {\it quandle} $X$ is a 
set with a binary operation $(x, y) \mapsto x * y$
satisfying the following conditions.
\begin{eqnarray*}
& &\mbox{\rm (1) \ }   \mbox{\rm  For any $x \in X$,
$x*x=x$.} \label{axiom1} \\
& & \mbox{\rm (2) \ }\mbox{\rm For any $y \in X$, the right multiplication $R_y: x \mapsto x*y$ is a bijection.} \label{axiom2} \\
& &\mbox{\rm (3) \ }  
\mbox{\rm For any $x,y,z \in X$, we have
$ (x*y)*z=(x*z)*(y*z). $} \label{axiom3} 
\end{eqnarray*}
\end{definition}
Notice that Axioms (2) and (3) can be combined to state that $R_y$ is an automorphism of the quandle $(X,*)$. The notions of quandle homomorphisms and isomorphisms are natural.  The following are some examples of quandles.
  \begin{itemize}
\item
Any non-empty set $X$ with the operation $x*y=x$ for any $x,y \in X$ is
a quandle called a  {\it trivial} quandle.

\item
Any group $G$ with the binary operation $x*y=yxy^{-1}$ is a quandle called a {\it conjugation quandle}.

\item
The binary operation $x*y=yx^{-1}y$ defines a quandle structure on any group $G$ called the \emph{Core} quandle of $G$ and denoted Core($G$).  If furthermore, $G=\mathbb{Z}_n$ is the abelian group of integers modulo $n$, where $n$ is a positive integer, then the operation is denoted by $x * y=-x+2y$  and the quandle is called the \emph{dihedral} quandle and denoted by $R_n$.

\item

A {\it generalized Alexander quandle} is defined  by a pair $(G, f)$ where 
$G$ is a  group and $f \in {\rm Aut}(G)$,
and the quandle operation is defined by 
$x*y=f(xy^{-1}) y $. 
If $G$ is abelian, this is called an {\it Alexander quandle}. 

\item
A function $\psi: X \times X \rightarrow A$ for an abelian group $A$ is 
called a {\it quandle $2$-cocycle}  \cite{CJKLS} if it satisfies 
$$ \psi (x, y)+ \psi(x*y, z) =\psi(x,z) +  \psi(x*z, y*z)$$
and $\psi(x,x)=0$ for all $x,y,z \in X$.
For a quandle $2$-cocycle $\psi$ the set $X \times A$ is a quandle under the operation 
 $(x, a) * (y, b)=(x*y, a+\psi(x,y))$ for $x, y \in X$, $a,b \in A$,
and it is called an {\it abelian extension} of $X$ by $A$, see  \cite{CENS}.

  \end{itemize}
  Given a quandle $X$, the subgroup $\Inn(X)$, of the automorphism group $\Aut(X)$, is the group generated by the right multiplications $R_x$.  When the natural action of $\Inn(X)$ on $X$ is transitive then $X$ is said to be a \emph{connected} quandle. \\
  \par

Throughout this article, $\mathbbm k$ denotes an integral domain with unity unless otherwise specified.\\

Let $(X, *)$ be a quandle and  $\mathbbm{k}[X]$ be the set of all formal expressions of the form $\sum_{x \in X }  \alpha_x e_x$, where $\alpha_x \in \mathbbm{k}$ such that all but finitely many $\alpha_x=0$.  The set $\mathbbm{k}[X]$ is a free $\mathbbm{k}$-module with basis $\{e_x \mid x \in X \}$ and admits a product given by 
 $$ \Big( \sum_{x \in X }  \alpha_x e_x \Big) \Big( \sum_{ y \in X }  \beta_y e_y \Big)
 =   \sum_{x, y \in X } \alpha_x \beta_y e_{x * y},$$
where $x, y \in X$ and $\alpha_x, \beta_y \in \mathbbm{k}$. This turns $\mathbbm{k}[X]$ into a ring called the {\it quandle ring} (or the {\it quandle algebra}) of $X$ with coefficients in $\mathbbm{k}$.  When $X$ is a non-trivial quandle then the quandle ring $\mathbbm{k}[X]$ is non-associative ring.  There is a natural identification of $X$ with a subset of $\mathbbm{k}[X]$ via the map sending $x$ to $e_x$. 
\par
The surjective ring homomorphism $\varepsilon: \mathbbm{k}[X] \rightarrow \mathbbm{k}$ given by $$\varepsilon \Big(\sum_{x \in X }  \alpha_x e_x \Big)=\sum_{x \in X }  \alpha_x$$
is called the {\it augmentation map}. The kernel $\Delta(X)$ of $\varepsilon $ is a two-sided ideal of $\mathbbm{k}[X]$, called the {\it augmentation ideal} of $\mathbbm{k}[X]$. 

 {Notation:}   Through the whole article, it is important to make a distinction between the product in the quandle, denoted by $*$, and the product, denoted by $\cdot$ or juxtaposition, in the quandle ring.  
 
\medskip

One of the interesting problems in the theory of quandle rings is the determination of idempotents, that is, elements $u \in \mathbbm{k}[X]$ such that $u^2 = u$.  Any element of $X$ is an idempotent, called a trivial idempotent. A more interesting case of quandle rings is when $\mathbbm{k} = \Z$.  
 For a full description of idempotents of  $\Z[X]$ one needs to understand the structure of $X$. 
If  $X = \coprod_i X_i$ is the disjoint union of quandles, then for the set of idempotents $I(\Z[X])$, the following inclusion holds $I(\Z[X]) \supseteq \coprod_i I(\Z[X_i])$.
Also, if 
 $X = \{ {1}, {2}, \ldots, {m} \}$ is the $m$ element trivial quandle, then 
 $$
 I(\Z[X]) = \{ e_{1} + \sum_{j=2}^m \alpha_j (e_{j} - e_{1}),~~\alpha_j \in \Z \}
  $$
(see Proposition 13.21 of \cite{BES}).  One of the interesting classes of quandles is the class of \emph{latin} quandles. Recall that a quandle is called {\it latin} if left multiplication by each element is a bijection.
In particular, the  dihedral quandles of odd orders  are latin.   
For the class of latin quandles, we have the following well-known conjecture

 \begin{conj}\label{conjid}  (\cite[Conjecture 13.19]{BES})  Let $X$ be a finite latin quandle, and let $u$  be  a non-zero idempotent of the integral quandle ring $\mathbb{Z}[X]$. Then
 
 \begin{enumerate}
 \item
 The element $u$ is a trivial idempotent;

\item
 $\varepsilon(u) = 1$. 
 \end{enumerate}
\end{conj}

\begin{remark}
Conjecture~\ref{conjid} does not extend to quandle rings over arbitrary fields. Indeed, if $\mathbbm{k}$ is a field, then part~(1) of the conjecture is disproved by item~(2) of the following proposition, whereas part~(2) is disproved by Proposition~\ref{pro4.9}.
\end{remark}


\begin{pro} (\cite[Proposition 12.18]{BES})
Let $X = \{ x_1, \ldots, x_n \}$ be a finite latin quandle. Then the following  assertions hold:
\begin{enumerate}
\item The element $w = e_{x_1} + \cdots + e_{x_n}$ lies in the center  $Z(\mathbbm{k}[X])$.
\item If $\mathbbm{k}$ is a field with characteristic not dividing $n$, then $\frac{1}{n}w$ is idempotent.
\end{enumerate}
\end{pro}

For quandle rings which have only trivial idempotents, it is easy to find the automorphism group of the quandle ring.

\begin{pro}\label{AutLatin}
	
	If a quandle ring $\mathbbm{k}[X]$ has only trivial idempotents,  then $\Aut(\mathbbm{k}[X]) = \Aut(X).$

\end{pro}

In the following sections, we study idempotents in the quandle rings of several special types of quandles.

\section{Infinite  Quandles} \label{infinitequandles}

The question of idempotents in quandle rings can be formulated not only for finite latin quandles but for infinite latin quandles and for semi-latin quandles.

As an example, consider the Core($\mathbb{Z})$ in which the operation is defined by the rule $a * b = 2 b - a$, for all $a, b \in \mathbb{Z}$. Let $L_a$, where $a \in \mathbb{Z}$, be the left multiplication operator; that is $L_a(x) = a * x$ for every $x \in \mathbb{Z}$. We see that $L_2 (x) = 2(x - 1)$.

Hence, $L_2$ is injective but not bijective, and therefore it is not invertible. Furthermore, $L_a^2(x)=4x-3a\neq x,$
so $L_a$ is not an involution. On the other hand, if $R_a$, where $a\in\mathbb{Z}$, denotes the right multiplication operator, then $R_a$ is an automorphism and satisfies $R_a^2=\operatorname{id}$.

Next, we study the structure of the augmentation ideal $\Delta(Core(\mathbb{Z}))$. To this end, we regard $\mathbb{Z}$ as the multiplicative group generated by $t$ and write
\[
\mathbb{Z}=\{\ldots,t^{-1},1,t,\ldots\}.
\]
Then the binary operation on $Core(\mathbb{Z})$ is given by
\[
t^i*t^j=t^{2j-i}.
\]
For simplicity, we denote the basis element $e_{t^i}$ of the quandle ring by $e_i$. Thus,
\[
e_i\cdot e_j=e_{2j-i}.
\]

The set $\mathbb{Z}$ is equipped with the linear order given by
\[
t^i<t^j \quad \text{ if} \quad i<j.
\]
Furthermore, for any $t^k \in \mathbb{Z}$ we have
$$
t^i * t^k = t^{2k-i} > t^j* t^k = t^{2k-j}.
$$
On the other side
$$
t^k * t^i = t^{2i-k} < t^k * t^j = t^{2j-k}.
$$
 Hence, this order is preserved under multiplication on the left and is not preserved under multiplication on the right. 
 The quandle $Core(\mathbb{Z})$ is left orderable but not right orderable  \cite{BPS1} (see also \cite[Propositions 11.3 and 11.5]{BES}).

Let us take
$$
u = \sum_{m \leq i \leq M} \alpha_i e_i,~~\alpha_i  \in \mathbb{Z},
$$
where we assume that  $\alpha_m \alpha_M \not= 0$. In this case we say that $\alpha_m e_m= \min (u)$, $\alpha_M e_M = \max (u)$, and $supp(u) \subseteq [m, M]$.
We want to find idempotents in $\mathbb{Z}[Core(\mathbb{Z})]$. For element $u$ we have 
$$
u^2 = \sum_{i, j} \alpha_i \alpha_j e_{2j-i}.
$$
Let us extend the linear order from $\mathbb{Z}$ to $\mathbb{Z}[Core(\mathbb{Z})]$ by the rule 
$$
\alpha e_a < \beta e_b ~\mbox{if and only if}~\alpha \beta \not= 0~\mbox{and}~a < b.
$$

The maximal and minimal elements in $u^2$ are defined by the rules
$$
\min(u^2) = \max(u) \min(u),~~\max(u^2) = \min(u) \max(u).
$$
Using this observation, we can prove the following result

\begin{thm}\label{Thm3.1}
If $\mathbbm{k}$ is an integral domain with unity,
then  $\mathbbm{k}[Core(\mathbb{Z})]$ has only trivial idempotents. 
\end{thm} 


\begin{proof}
Let 
$$
u = \alpha_m e_m + \ldots + \alpha_M e_M, ~~\alpha_m, \alpha_M \in \mathbbm{k} \setminus \{0 \}
$$
be a non-zero element in $\mathbbm{k}[Core(\mathbb{Z})]$ and $\min(u) = \alpha_m e_m$, $\max(u) = \alpha_M e_M$.

If $m = M$, then $u = \alpha_m e_m$ is an idempotent if and only if $\alpha_m = 1$. This idempotent is trivial.

Let $m < M$. Then
$$
\min(u^2) = \max(u) \min(u) = \alpha_M e_M \alpha_m e_m = \alpha_M  \alpha_m e_{2m-M},
$$
$$
\max(u^2) = \min(u) \max(u) = \alpha_m e_m \alpha_M e_M  = \alpha_m  \alpha_M e_{2M-m}.
$$
Let us show that the following inequalities hold
$$
\min(u^2) < \min(u) < \max(u) < \max(u^2).
$$
The first inequality means
$$
\alpha_M  \alpha_m e_{2m-M} < \alpha_m e_{m} \Leftrightarrow 2m-M < m \Leftrightarrow m < M.
$$

The second inequality holds by definition.  The third  inequality means
$$
\alpha_M e_{M} < \alpha_m  \alpha_M e_{2M-m}  \Leftrightarrow M < 2M - m \Leftrightarrow m < M.
$$

From the proved inequalities follows that $u^2 \not= u$.
\end{proof}

It is easy to check that  if $X$ is an infinite field, then $Core (X)$ with operation $a * b = 2 b - a$ is a latin quandle.

\begin{qst}
Is it possible to replace $\Z$ in Theorem \ref{Thm3.1} with an arbitrary  ordered (abelian) group $G$? 

\end{qst}

\begin{qst}
Is it true that for any unital ring $\mathbbm {k}$ without zero-divisors  and an infinite field $X$ the quandle ring $\mathbbm {k} [Core(X)]$ has only trivial idempotents?
\end{qst}

\begin{qst}
What can we say about  $\Delta^2(Core(\mathbbm {Z}))$? In particular, is it true that $E_1 = e_1 - e_0$ lies in $\Delta^2(Core(\mathbbm {Z}))$?
\end{qst}

\subsection{Commutator Subalgebra}

Let $A = \mathbbm {Z}[Core(\mathbbm {Z})]$ and $A'$ be its commutator subalgebra. That is the subalgebra generated by ring commutators $[u, v] = u v - v u$, $u, v \in A$. It is easy to see that $A' $ is generated by commutators
$$
[e_i, e_j] = e_i e_j - e_j e_i = e_{i*j} -  e_{j*i} =  e_{2j -i} -  e_{2 i -j}.
$$
Hence $A' \subseteq \Delta(Core(\mathbbm {Z}))$. One can naturally ask if the other 
inclusion is true? 

We have the following

\begin{pro}
The following inclusion holds  $ \mathbbm {Z}[Core(3\mathbbm {Z})]  \subseteq A'$. 
\end{pro}

\begin{proof}
As we know, $\Delta(Core(\mathbbm {Z}))$ is generated by elements $E_a = e_a - e_0$ which can be written as $E_a = [e_{\frac{a}{3}}, e_{\frac{2a}{3}}]$. Hence, if $a$ is divisible by $3$, then this element lies in the commutator subalgebra. 
\end{proof}


\section{Idempotents in Quandle Rings of Dihedral Quandles} \label{Idemp-Dihedral}

Let $X$ be a quandle and $\mathbbm {k}[X]$ be its quandle ring. To find idempotents in $\mathbbm {k}[X]$ we can write 
$$
\mathbbm {k}[X] = \mathbbm {k} e_0 + \Delta_{\mathbbm {k}}(X)
$$
and any element $u$ of this algebra has the form
$$
u = \alpha e_0 + \delta,~~\alpha \in \mathbbm {k},~~e_0 \in X,~~\delta \in \Delta_{\mathbbm {k}}(X).
$$
Hence,
$$
u^2 = (\alpha e_0 + \delta)^2 = \alpha^2  e_0 + \alpha (e_0 \delta +  \delta e_0) + \delta^2.
$$
Since, $\Delta_{\mathbbm {k}}(X)$ is a two-sided ideal, $ \alpha (e_0 \delta +  \delta e_0) \in \Delta_{\mathbbm {k}}(X)$, and $u$ is an idempotent if and only if
$$
 \alpha^2  =  \alpha, ~~\alpha (e_0 \delta +  \delta e_0) + \delta^2 =   \delta.
$$
From the first equality follows that $ \alpha = 1$, or $ \alpha = 0$. In the first case, we have $\varepsilon(u) = 1$ and in the second case $\varepsilon(u) = 0$. Hence, if we want to find idempotent $u$ such that $\varepsilon(u) = 0$ we must find element $u$ which lies in  $\Delta_{\mathbbm {k}}(X)$.

Suppose that $u$ is an idempotent in some groupoid. Then it is easy to see that subgroupoid, which is generated by $u$ contains only one element. In particular,
$$
u^n = (\ldots((uu)u)\ldots u) = u.
$$ 

In $\mathbbm {Z}[R_n]$ put
$$
E_i = e_i - e_0,~~i = 1, 2, \ldots, n-1.
$$
These elements form a basis of $\Delta(R_n)$. The products of these elements are described by the following lemma, whose proof is a straight computation.

\begin{lemma} \label{E}
The following equalities hold
$$
E_i E_j = E_{2j-i} - E_{n-i} - E_{2j},
$$
where all indices are reduced modulo $n$ and we assume $E_0 = 0$. In particular,
$$
E_i^2= E_{i} - E_{n-i} - E_{2i}.
$$
\end{lemma}

At first, suppose that $X = R_3 = \{ 0, 1, 2 \}$ is the 3-element  dihedral quandle and $\mathbbm {k} = \Z$. The elements $E_1$ and $E_2$  form a basis of $\Delta(R_3)$.
In the following proposition, we  consider powers of augmented ideals $\Delta = \Delta(R_3)$ and denote $\Delta^{k+1} = \Delta^k \Delta$, $k = 1, 2, \ldots$.

\begin{pro}\label{pro4.8}
For the dihedral quandle $R_3$, the powers of the augmentation ideal are given by:

\[
\Delta(R_3)^{2k+1}
= \langle 3^k E_1,\; 3^k E_2 \rangle,
\quad k \ge 0,
\]
and
\[
\Delta(R_3)^{2\ell}
= \langle 3^{\ell-1}(E_1+E_2),\; 3^\ell E_2 \rangle,
\quad \ell \ge 1.
\]
 


\end{pro}

\begin{proof}

Since $R_3$ is a commutative quandle, we see that $\Delta^2$ is generated by elements $(E_1)^2$, $E_1 E_2$, $(E_2)^2$. We find that
$$
(E_1)^2 = E_1 - 2 E_2,~~E_1 E_2 = - E_1 -  E_2,~~(E_2)^2 = E_2 - 2 E_1.
$$
It is easy to see that elements $E_1 + E_2$ and $3 E_2$ form a basis of $\Delta^2$.

Hence, we have a basis of induction.  By induction hypothesis we assume that for the odd powers,   $\Delta(R_3)^{2(k-1) + 1}$  has a basis: $\{ 3^{k-1} E_1, 3^{k-1} E_2 \}$. Then $\Delta(R_3)^{2k + 1} = \Delta(R_3)^{2(k-1) + 1} \cdot \Delta^2$ is generated by elements
$$
3^{k-1} E_1 \cdot (E_1 + E_2) = - 3^{k} E_2,~~3^{k-1} E_1 \cdot 3 E_2 = - 3^{k} (E_1 + E_2),
$$
$$
3^{k-1} E_2 \cdot (E_1 + E_2) =  - 3^k E_1,~~
3^{k-1} E_2 \cdot 3 E_2 =  3^k (-2 E_1 + E_2),
$$
and it is easy to check that the set  $\{ 3^{k} E_1, 3^{k} E_2 \}$ form a basis of  $\Delta(R_3)^{2k + 1}$.

 By  induction hypothesis we assume for the even powers,   $\Delta(R_3)^{2(l-1)}$, $l > 1$,  has a basis: $\{ 3^{l-2} (E_1 + E_2), 3^{l-1} E_2 \}$. Then $\Delta(R_3)^{2l} = \Delta(R_3)^{2(l-1)}\cdot \Delta^2$ is generated by the elements
$$
3^{l-2} (E_1 + E_2) \cdot (E_1 + E_2) = - 3^{l-1} (E_1 + E_2),~~3^{l-2} (E_1 + E_2) \cdot 3 E_2 = - 3^{l} E_1,
$$
$$
3^{l-1} E_2  \cdot (E_1 + E_2) =  - 3^l E_1,~~3^{l-1} E_2  \cdot 3 E_2 =  3^l (-2 E_1 + E_2),
$$
and it is easy to check that the set  $\{ 3^{l-1}(E_1 + E_2), 3^{l} E_2 \}$ form a basis of  $\Delta(R_3)^{2l}$.

\end{proof}

We have the following immediate corollary.

\begin{cor}
$\mathbbm {Z}[R_3]$ does not contain non-zero idempotent $u$ such that  $\varepsilon(u) = 0$.
\end{cor}

We have the following immediate corollary.

\begin{pro}\label{pro4.9}
If $\mathbbm {k}$ is a field of $char(\mathbbm {k}) \not=3$,  then the element
$$
u = - \frac{1}{3} \left( E_1 + E_2\right)
$$
is an idempotent of $\mathbbm {k}[R_3]$ such that  $\varepsilon(u) = 0$.
\end{pro}


 In Corollary 13.42 of \cite{BES} one can find idempotents in $\mathbbm {k}[R_{2n}]$.
For $\mathbbm {Z} [R_4]$ we have the following

\begin{pro} 

\begin{enumerate}
\item

The quandle ring  $\mathbbm {Z} [R_4]$ 
does not contain idempotents for which the value of the augmentation $\varepsilon$ is equal to $0$.

\item
 The quandle ring  $\mathbbm {Z} [R_4]$ 
contains the idempotents
$$
e_0 - 2 \beta  E_1 +  \beta  E_2 -  2 \beta  E_3,~~
e_0 +  \beta_1  E_1 +  \beta_2  E_2 + (1 - \beta_1)  E_3,\;  \text{where} \;  \beta, \beta_1,  \beta_2 \in \mathbbm {Z},
$$
for which the value of $\varepsilon$ is equal to $1$.
\end{enumerate}
\end{pro}

\begin{proof}
For item (1), let
$$
u = \a E_1 + \b E_2 + \g E_3 \in \Delta(R_4).
$$
This element is a non-zero  idempotent if and only if the following system has non-zero integer solutions
$$
\begin{cases}
\a = \a^2 - \g^2,\\
\b = -(\a^2 + \g^2 + 2 (\a \g + \a \b + \b \g)), \\
\g = -\a^2 + \g^2.
\end{cases}
$$
From the first and the third equations follows $\a = -\g = 0$. Then from the second equation  we get $\b =  0$.

Item (2) follows from  \cite[Corollary 13.42]{BES}.
\end{proof}

 Powers of fundamental ideal for the quandle ring $\mathbbm {Z} [R_4]$  were computed in Proposition 12.21 of \cite{BES} which we state here.

\begin{pro}\label{ass-graded-2}
The following statements hold for $R_4$:
\begin{enumerate}
\item $\Delta^2(R_4)$ is generated as an abelian group by $\{E_1 - E_2 - E_3,~ 2 E_2\}$
and $\Delta(R_4) / \Delta^2(R_4) \cong \mathbb{Z} \oplus \mathbb{Z}_2$.
\item If $k > 2$, then $\Delta^k(R_4)$ is generated as an abelian group by $\{2^{k-2}(E_1 - E_2 - E_3), ~2^{k-1} E_2\}$
and $\Delta^{k-1}(R_4) / \Delta^k(R_4) \cong \mathbb{Z}_2 \oplus \mathbb{Z}_2$.
\end{enumerate}
\end{pro}

Let us consider the quandle ring  $\mathbbm {Z}[R_5]$.
Its augmented ideal $\Delta(R_5)$ is generated by elements 
 $E_1, \ldots, E_4$. Using Lemma \ref{E} we have
$$
E_1^2 = E_1 - E_2 - E_4,~~E_2^2 = E_2 - E_3 - E_4,~~E_3^2 = E_3 - E_2 - E_1,~~E_4^2 = E_4 - E_1 - E_3,
$$
$$
E_1 E_2 =  E_3 - 2 E_4,~~E_2 E_1 =  - E_3 - E_2,~~E_1 E_3 =  - E_4 - E_1,~~E_3 E_1 = E_4 - 2 E_2,~~
$$
$$
E_1 E_4 =  E_2 -  E_4 - E_3,~~E_4 E_1 =  E_3 - E_1 - E_2,~~E_2 E_3 =   E_4 - E_3 - E_1,~~E_3 E_2 = E_1 - E_4 - 2 E_2, 
$$
$$
E_3 E_2 = E_1 - E_4 - E_2,~~E_2 E_4 =  E_1 -  2 E_3,~~E_4 E_2 =   - E_1 - E_4,~~E_3 E_4 =    - E_2 - E_3,~~E_4 E_3 =  E_2 - 2 E_1.
$$
From these computations we obtain the following 
\begin{lemma} \label{sqr}
The four elements
$$
F_1:=E_1 - E_2 - E_4,~~F_2:=E_2 - E_3 - E_4,~~ F_3:=E_3 + 3  E_4\;\;  \mbox{and} \;\; F_4:=5 E_4.
$$
form a basis of $\Delta^2(R_5)$.
\end{lemma}

 \begin{conj}
For any $a \in \Delta^{2n+1}(R_{2n+1})$, its image under the augmentation map (in $\Delta^{2n+1}(R_{2n+1})$) has the value $\varepsilon(a) \in (2n+1) \mathbb{Z}.$ 
\end{conj}

\begin{pro}
The quandle ring  $\mathbbm {Z}[R_5]$ has a non-trivial idempotent for which the value $\varepsilon$ is equal to $1$ if and only if 
the following system has non-zero integer  solutions in which more than one component is non-zero
$$
\begin{cases}
\alpha_1 = \alpha_3 + \alpha_4 +   \alpha_1^2 - \alpha_3^2 - \alpha_4^2 - \alpha_1\alpha_3  - \alpha_1 \alpha_4  - 2 \alpha_3 \alpha_4,    \\ 
\alpha_2 = \alpha_1 + \alpha_3   - \alpha_1^2 + \alpha_2^2 - \alpha_3^2 - \alpha_1\alpha_2  - 2 \alpha_1 \alpha_3  -  \alpha_2 \alpha_3,    \\ 
\alpha_3 = \alpha_2 + \alpha_4   -  \alpha_2^2 + \alpha_3^2 - \alpha_4^2 - \alpha_2\alpha_3  - 2 \alpha_2 \alpha_4  -  \alpha_3 \alpha_4,    \\ 
\alpha_4 = \alpha_1 + \alpha_2 - \alpha_1^2 - \alpha_2^2 + \alpha_4^2 - 2 \alpha_1\alpha_2  -  \alpha_1 \alpha_4  -  \alpha_2 \alpha_4.
\end{cases}
$$
\end{pro}

\begin{proof}
Any element $u$ in $\mathbbm {Z}[R_5]$ for which $\varepsilon(u) = 1$ has a form 
$$
u = e_0 + \delta = e_0 + \alpha_1 E_1 + \alpha_2 E_2 + \alpha_3 E_3 + \alpha_4 E_4
$$
for some integers $\alpha_i$. Let us find $u^2$. We have $u^2 = e_0 + (\delta e_0 + e_0 \delta) + \delta^2$. Using straightforward computations we get
$$
\delta e_0 + e_0 \delta = (\alpha_3  + \alpha_4)  E_1 +  (\alpha_1  + \alpha_3)  E_2 +  (\alpha_2  + \alpha_4)  E_3 +  (\alpha_1  + \alpha_2)  E_4.
$$
Further,
$$
\delta^2 = (\alpha_1^2 - \alpha_3^2 - \alpha_4^2 - \alpha_1\alpha_3  - \alpha_1 \alpha_4  - 2 \alpha_3 \alpha_4) E_1 +
 (- \alpha_1^2 + \alpha_2^2 - \alpha_3^2 - \alpha_1\alpha_2  - 2 \alpha_1 \alpha_3  -  \alpha_2 \alpha_3) E_2 +
$$
$$ 
+ (- \alpha_2^2 + \alpha_3^2 - \alpha_4^2 - \alpha_2\alpha_3  - 2 \alpha_2 \alpha_4  -  \alpha_3 \alpha_4) E_3 +
  (- \alpha_1^2 - \alpha_2^2 + \alpha_4^2 - 2 \alpha_1\alpha_2  -  \alpha_1 \alpha_4  -  \alpha_2 \alpha_4) E_4.
$$
Hence, from the equation $u^2 = u$, we obtain the system of equations.
\end{proof}

\section{Commutative Quandles, Powers of Augmentation Ideals and Idempotents} \label{PowerIdealCommutative}

A quandle $(X, *)$ is said to be commutative, if $x * y = y * x$ for all $x, y \in X$. It is easy to see that any commutative quandle is a latin quandle.    Notice that this implies also that commutative quandles are faithful meaning that the map $x \mapsto R_x$ is injective.  This is clear since for any $z$, $R_x(z)=R_y(z)$ implies $x*z=y*z$ and thus $x=y$.

\begin{exs}
\begin{enumerate}

\item
 On the set $\mathbb{Z}_{2n+1}$ define the following operation,
$$
x * y = (n+1) (x + y),~~x, y \in \mathbb{Z}_{2n+1}.
$$
Then $(\mathbb{Z}_{2n+1}, *)$ is a commutative quandle. We will denote it by $C_{2n+1}$. If $n=1$, then $C_3 $ is the dihedral quandle $ R_3$ since $2y-x=2(x+y)$ mod $3$.

\item
The set of  dyadic rationals,
$$
\left\{ \frac{m}{2^k} ~|~m \in \mathbb{Z},  k \in \{ 0, 1, 2, \ldots \}   \right\}
$$
with the operation $a * b = \frac{1}{2} (a + b)$ is a commutative quandle.
 

\item 

As remarked by Luc Ta \cite{Ta2026Medial}, a conjugation quandle  $Conj(G)$ is commutative if and only if the group $G$ is trivial,  and a core quandle $Core(G)$ is commutative if and only if $G$ has exponent 3.

 \end{enumerate}

\end{exs}

\noindent
From these examples we see that there are many commutative quandles. We can formulate a question on the classification of finite commutative quandles. The first step towards an answer to this question is the following

\begin{pro}
If  $(X,*)$ is a finite commutative quandle then it contains an odd number of elements.
\end{pro}

\begin{proof}
Let's assume that the cardinality of $X$ is greater than $2$.  Now fix an element  $x \in X$.  Given an element $y \neq x$, by the second quandle axiom, there exists a unique $z \in X$ such that $z*y=x$.  Now since $x \neq y$ and $X$ is a commutative (then latin) quandle, the elements $x,y,z$ are pairwise disjoint.   Now assume we have a family of elements $x, y_1, z_1, \ldots y_n, z_n$ such that for all $i$ with $1\leq i \leq n$, $z_i*y_i=x$.  Now let $y_{n+1} \not\in \{ x, y_1, z_1, \ldots y_n, z_n \} $.  There exists a unique $z_{n+1} \in X$ such that $z_{n+1} *y_{n+1} =x$.  Now we \emph{claim} that  $z_{n+1} \not\in \{ x, y_1, z_1, \ldots y_n, z_n, y_{n+1} \}.$ It is clear that $z_{n+1} \neq y_{n+1}$, otherwise $y_{n+1} =x$ and thus a contradiction.  Now if $z_{n+1} =z_i$ for $1 \leq i \leq n$ then $y_{n+1} =y_i$, thus a contradiction.  Lastly, if $z_{n+1} =y_i$ for $1 \leq i \leq n$ then we get $y_i*y_{n+1}=y_i*z_i=x$ giving that $z_i=y_{n+1}$, and thus a contradiction.   Then  the cardinality of $X$ is an odd number.
\end{proof}

From straight calculations, we have the following
\begin{lemma}
In a commutative quandle $(X, *)$ for any $x, y, z \in X$,  it holds that,
\begin{enumerate}
\item

 $(x * y) * z = (y * x) * z = z * (y * x) = z * (x * y)$, in particular, $(x * y) * x = x * (x * y) = x * (y * x)$; 
 
 \item

 $R_x^{-1} R_y R_x = R_y^{-1} R_x R_y$, where we consider the actions from the left to the right, i.e. 
$(a) (R_x^{-1} R_y R_x) = (((a) R_x^{-1}) R_y)R_x$ for any $a \in X$.

\item
 $R_x^{-1} R_y = [R_y, R_x^{-1}]$, where $[a, b] = a^{-1} b^{-1} a b$ is the group commutator. 
 \end{enumerate}
\end{lemma}


The following properties of commutative quandles are evident.

\begin{pro}
\begin{enumerate}

\item

Let \(2n+1\) be an odd integer, where \(n>0\). Let
\(\varphi \colon \mathbb{N} \to \mathbb{N}\) denote Euler's 
function (i.e. $\varphi(m)$ is the number of units in $\mathbb{Z}_m$). Then the commutative quandle \((C_{2n+1}, *)\) is a
\(\varphi(2n+1)\)-quandle; that is,
$
R_y^{\varphi(2n+1)}=\operatorname{id}_{C_{2n+1}}
$
for every \(y\in C_{2n+1}\). Equivalently,
$
x *^{\varphi(2n+1)} y = x
$
for all \(x,y\in C_{2n+1}\).

\item
 In the quandle $(C_{2n+1}, *),$ the inverse operation is defined by the rule $x \bar{*} y = 2 x - y$. 
 
\item
 Direct sum of commutative quandles is commutative.
\end{enumerate}

\end{pro}

\begin{proof}

Item (1) follows from the following.
First, a direct computation yields the equation
\[
R^{\varphi(2n+1)}_y(x) = (n+1)^{\varphi(2n+1)} x + [\sum_{k=1}^{\varphi(2n+1)}(n+1)^k]y.
\]

Now, since $n+1$ is invertible in $C_{2n+1}$, by Euler theorem we get that $(n+1)^{\varphi(2n+1)}=1.$   Now we claim that the sum $$ \sum_{k=1}^{\varphi(2n+1)}(n+1)^k=0$$ and thus item (1) will hold.  The following is a short proof that the sum is zero.

Let $
m=2n+1$, we have $n+1= 2^{-1} \in C_m$, thus the sum
\[
S=\sum_{k=1}^{\varphi(m)}(n+1)^k=\sum_{k=1}^{\varphi(m)}(2^{-1})^k.
\]

By Euler's theorem, $
2^{\varphi(m)}=1$,
and hence
$
(2^{-1})^{\varphi(m)} = 1.$ Furthermore, $
2^{-1}-1=n.$
Since $\gcd(n,m)=1,$
the element $2^{-1}-1$ is invertible in $\mathbb{Z}_m$. Therefore, the geometric series formula yields
\[
S
=(2^{-1})\,
\frac{(2^{-1})^{\varphi(m)}-1}{2^{-1}-1}.
\]

That is,
\[
\sum_{k=1}^{\varphi(2n+1)}(n+1)^k = 0.
\]


 Items (2) and (3) are straightforward computations.
\end{proof}

\medskip

\begin{qst}
When $2n+1$ is a prime number, it is interesting to find idempotents of the quandle ring of the commutative quandle $C_{2n+1}$. 
\end{qst}

 Denote by $\{ a_0, a_1, \ldots, a_{2n} \}$ the basis of the quandle ring $\Z[C_{2n+1}]$. Denote $f_i = a_i - a_0$, $i = 1, 2, \ldots , 2n$. These elements form a basis of 
 $ \Delta(C_{2n+1})$.

 The following proposition is analogous to Proposition~\ref{pro4.9}. 

\begin{pro}
If $\mathbbm {k}$ is a field of $char(\mathbbm {k}) \not=2n+1$, then in the quandle ring $\mathbbm {k}[C_{2n+1}]$, the element
$$
u = - \frac{1}{2n+1} \left( f_1 + \cdots + f_{2n}\right)
$$
is idempotent.
\end{pro}

\begin{proof}
Let $w=a_0+a_1+ \cdots+a_{2n} \in \mathbbm {k}[C_{2n+1}]$.  Notice that because the quandle is latin this vector is invariant by both right and left multiplications.
Now we can re-write $u = - \frac{1}{2n+1} \left( f_1 + \cdots + f_{2n}\right)$ as $u = a_0- \frac{1}{2n+1} w$ and thus $u^2=u$.
\end{proof}

Let us find the powers of the augmentation ideal $\Delta(C_5)$ of $\Z[C_5]$.  The elements  $\{ a_0$, $a_1$, $a_2$, $a_3, a_4 \}$ are the  basis elements of $\Z[C_5]$
 with the product $a_i \cdot a_j = a_{3(i+j)}$, where the sum in the additive cyclic group of order 5.
The augmentation ideal $\Delta(C_5)$ has a basis $f_1, f_2, f_3, f_4$ and  it is not difficult to check that  $\Delta^2(C_5)$ has a basis
$$
f_1 - 2 f_3,~~f_2 -  f_3 - f_4,~~f_3 - 2 f_4,~~5 f_4.
$$
Furthermore, $\Delta^3(C_5) = \Delta^2(C_5) \cdot \Delta(C_5)$ has a basis
$$
f_1 - f_3 - 2 f_4,~~f_2 + 2  f_3 - 2 f_4,~~5 f_3,~~5 f_4,
$$
and  $\Delta^4(C_5) = \Delta^3(C_5) \cdot \Delta(C_5)$ has a basis
$$
f_1 + f_2 +  f_3 + f_4,~~5 f_2,~~5 f_3,~~5 f_4.
$$

Using induction on $l$, it is not difficult to formulate the following proposition giving the general powers~$\Delta^l(C_5)$. 

\begin{pro}\label{CommutativeZ5}
Let  $C_5$ be the commutative 5-element quandle.  The powers of the  augmentation ideal $\Delta = \Delta(C_5) =\langle f_1, \cdots, f_4 \rangle$ are given by the formulas:

\begin{eqnarray}
\Delta^{4k} &=& \langle  5^{k-1}(f_1+ \cdots +f_4), \; 5^kf_2, \; 5^kf_3, \; 5^kf_4   \nonumber \rangle,~~k = 1, 2, \ldots \\
\Delta^{4k+1} &=& \langle  5^kf_1, \cdots, \; 5^kf_4   \nonumber \rangle,~~k = 0, 1, \ldots \\
\Delta^{4k+2} &=& \langle  5^k(f_1+f_4), \; 5^k(f_2+2f_4), \; 5^k(f_3+3f_4),  \; 5^{k+1}f_4   \nonumber \rangle,~~k = 0, 1, \ldots \\
\Delta^{4k+3} &=& \langle  5^k(f_1+4f_3+3f_4), \;  5^k(f_2+2f_3+3f_4),\; 5^{k+1}f_3, \; 5^{k+1}f_4,  \nonumber \rangle,~~k = 0, 1, \ldots.
\end{eqnarray}
\end{pro} 

The proof of the following proposition is similar.

\begin{pro}\label{CommutativeZ7}
Let  $C_7$ be the commutative 7 element quandle.  The powers of the  augmentation ideal $\Delta=\Delta(C_7) =\langle f_1, \cdots, f_6 \rangle$ are given by the formulas:

\begin{eqnarray}
\Delta^{6k} &=& \langle  7^{k-1}(f_1+ \cdots +f_6), \; 7^kf_2, \; \cdots, \; 7^kf_6   \nonumber \rangle,~~k = 1, 2, \ldots\\
\Delta^{6k+1} &=& \langle  7^kf_1, \cdots, \; 7^kf_6   \nonumber \rangle,~~k = 0, 1, \ldots\\
\Delta^{6k+2} &=& \langle  7^{k}(f_1+f_6), \; 7^{k}(f_2+2f_6), \; 7^{k}(f_3+3f_6), \nonumber\\
&& 7^{k}(f_4+4f_6), \; 7^{k}(f_5+5f_6),   \; 7^{k+1}f_6   \nonumber \rangle,~~k = 0, 1, \ldots\\
\Delta^{6k+3} &=& \langle  7^k(f_1+6f_5+3f_6), \;  7^k(f_2+4f_5+f_6),\; 7^k(f_3+f_5+f_6), \nonumber \\
&& 7^k(f_4+4f_5+3f_6),\; 7^{k+1}f_5, \; 7^{k+1}f_6  \nonumber \rangle,~~k = 0, 1, \ldots.\\
\Delta^{6k+4} &=& \langle  7^k(f_1+f_4+3f_5+6f_6), \; 7^k(f_2+2f_4+6f_5+6f_6),  \nonumber \\
&& 7^k(f_3+3f_4+6f_5+3f_6),  \; 7^kf_4, \; 7^kf_5, \; 7^kf_6   \nonumber \rangle,~~k = 0, 1, \ldots.\\
\Delta^{6k+5} &=& \langle  7^k(f_1+6f_3+5f_4+4f_5+3f_6), \;  7^k(f_2+2f_3+3f_4+4f_5+5f_6), \nonumber \\
&& 7^{k+1}f_3, \; 7^{k+1}f_4,  \; 7^{k+1}f_5,  \; 7^{k+1}f_6   \nonumber \rangle,~~k = 0, 1, \ldots.
\end{eqnarray}

\end{pro} 

From these propositions we get the following

\begin{cor}
The quandle rings  $\Z[C_5]$ and $\Z[C_7]$   do not contain non-zero idempotents $u$ such that $\varepsilon(u) =0$.
\end{cor}


Suppose now that $u$ is an idempotent in  $\Z[C_5]$ for which $\varepsilon(u) =1$. Then 
$$
u = a_0 + \delta = a_0 + \b_1 f_1 + \b_2 f_2 + \b_3 f_3 + \b_4 f_4
$$
for some integers $\b_i$. Then $u^2 = a_0 + 2 \delta a_0 + \delta^2$. It is not difficult to see that
$$
\delta a_0 =  \b_1 f_3 + \b_2 f_1 + \b_3 f_4 + \b_4 f_2
$$
and
$$
\delta^2 = (\b_1^2 - 2 \b_2^2 - 2 \b_1 \b_2 - 2 \b_2 \b_3 - 2 \b_2 \b_4 + 2 \b_3 \b_4) f_1 + 
$$
$$
+ (\b_2^2 - 2 \b_4^2 + 2 \b_1 \b_3 - 2 \b_1 \b_4 - 2 \b_2 \b_4 - 2 \b_3 \b_4) f_2 
+ (- 2 \b_1^2 + \b_3^2 - 2 \b_1 \b_2 - 2 \b_1 \b_3 - 2 \b_1 \b_4 + 2 \b_2 \b_4) f_3 +
$$
$$
+ (- 2 \b_3^2 + \b_4^2 + 2 \b_1 \b_2 - 2 \b_1 \b_3 - 2 \b_2 \b_3 - 2 \b_3 \b_4) f_4.
$$

This element $u$ is an idempotent if and only if $2 \delta a_0 - \delta = -\delta^2$. This is equivalent to the system
$$
\begin{cases}
2 \b_2 - \b_1 = -\b_1^2 + 2 \b_2^2 + 2 \b_1 \b_2 + 2 \b_2 \b_3 + 2 \b_2 \b_4 - 2 \b_3 \b_4,    \\ 
2 \b_4 - \b_2 = -\b_2^2 + 2 \b_4^2 - 2 \b_1 \b_3 + 2 \b_1 \b_4 + 2 \b_2 \b_4 + 2 \b_3 \b_4,    \\ 
2 \b_1 - \b_3 = -\b_3^2 + 2 \b_1^2 + 2 \b_1 \b_2 + 2 \b_1 \b_3 + 2 \b_1 \b_4 - 2 \b_2 \b_4,    \\ 
2 \b_3 - \b_4 = -\b_4^2 + 2 \b_3^2 - 2 \b_1 \b_2 + 2 \b_1 \b_3 + 2 \b_2 \b_3 + 2 \b_3 \b_4.
\end{cases}
$$

\begin{qst}
Is it true that this system does not have  integer  solutions in which more than one component are non-zero? 
\end{qst}

 
\section{$m$-almost Latin Quandles}\label{2-almost}

As we have seen above, if a quandle ring has only trivial idempotents, then the automorphism group of the quandle ring is equal to the automorphism group of the quandle.  In this section we consider a quandle ring of some generalizations of latin quandle  and find its idempotents and automorphisms.

Almost latin quandles are a generalization of latin quandles.  The following definition can be found in \cite{PY}.
  \begin{definition} 
Let $m\geq 1$ be an integer.  A quandle $X$ is said to be \index{$m$-almost latin quandle}{\it $m$-almost latin} if it satisfies the following conditions:
\begin{enumerate}
\item The stabilizer $Stab(a) =\{ x \in X | \; R_x(a)=a \}$ has order $m$  for each $a \in X$.
\item The equation $L_a(x)=b$ has a unique solution for each $b \in X \setminus Stab(a)$.
	\end{enumerate}
\end{definition}
Notice that when $m=1$, an $m$-almost latin quandle is simply a latin quandle.

Let us consider some example of 2-almost latin quandle (see \cite[Example 13.17]{BES}).
This quandle  $X= \{1, 2, \ldots, 6 \}$ is the involutory  connected quandle of order 6 given in terms of its right multiplications as follows:
\begin{eqnarray*}
&& R_1=(3\;5)(4\;6), \quad R_2=(3\;6)(4\;5), \quad R_3=(1\;5)(2\;6),\\
&& R_4=(1\;6)(2\;5), \quad R_5=(1\;3)(2\;4), \quad R_6=(1\;4)(2\;3).
\end{eqnarray*}
It is easy to see that $X$ is the disjoint union of trivial subquandles $\{ 1, 2 \}$, $\{ 3, 4 \}$, $\{ 5, 6 \}$, and then $\Z[X]$ has  non-trivial idempotents
$$
\alpha e_1 + (1 - \alpha) e_2,~~\beta e_3 + (1 - \beta) e_4,~~\gamma e_5 + (1 - \gamma) e_6,~~ \alpha, \beta, \gamma \in \Z.
$$ 
Let us show that any idempotent of  $\Z[X]$ has a form from this list. More accurately, the following proposition holds.

\begin{pro} \label{IdX}
Any idempotent of  $\Z[X]$  is of one of the following $3$ forms:  $\alpha e_1 +(1-\alpha)e_2$, $\beta e_3+(1-\beta)e_4$ and $\gamma e_5 +(1-\gamma)e_6$, where $\alpha, \beta, \gamma \in \mathbb{Z}$.
\end{pro} 

\begin{proof}
Any element of $\Z[X]$ has a form
 $$
\alpha_1 e_1 + \alpha_2 e_2 + \beta_1 e_3 +\beta_2 e_4 + \gamma_1 e_5 +  \gamma_2 e_6,~~ \alpha_i, \beta_i, \gamma_i \in \Z.
 $$

Let $R_3=\{ \overline 0, \overline 1,\overline 2\}$ be the dihedral 3 element quandle. As we know this quandle is commutative. Also, it is easy to check that the map 
$f: X \rightarrow R_3$, which acts by the rules
$$
f(1) = f(2) = \overline 0,~~f(3) = f(4) = \overline 1,~~f(5) = f(6) = \overline 2,
$$
is a homomorphism of quandles. By linearity this homomorphism can be extended to a homomorphism of quandle rings $F: \Z[X] \rightarrow \Z[R_3]$. The kernel of this homomorphism  consists of the following elements
$$
\ker (F) = \{ \alpha (e_1 - e_2) + \beta(e_3 - e_4) + \gamma (e_5 - e_6) ~ | ~  \alpha, \beta, \gamma \in \Z \}.
$$

Suppose that $u$ is a non-zero idempotent in $\Z[X]$. Since under a homomorphism any   idempotent goes to an  idempotent, then $F(u)$ is an idempotent in $\Z[R_3]$.
By Proposition 4.3 of \cite{BPS1}  $F(u) = 0$ or $F(u) \in \{ e_{\overline 0}, e_{\overline 1}, e_{\overline 2} \}$ is a trivial idempotent.

Suppose at first that $F(u) = 0$. It means that $u$ lies in the kernel $\ker (F)$. Let
$$
u = \alpha (e_1 - e_2) + \beta(e_3 - e_4) + \gamma (e_5 - e_6) 
$$
for some $\alpha, \beta, \gamma \in \Z$. Then
$$
u^2 = 4 \alpha \beta (e_5 - e_6) + 4 \alpha \gamma (e_3 - e_4) +4 \beta \gamma (e_1 - e_2)
$$
and it is easy to see that $u^2 = u$ if and only if $\alpha =  \beta =  \gamma = 0$.

Further, if $F(u) = e_{\overline j}$, $j \in \{ 0, 1, 2 \}$, then 
$$
u_i = e_i + \alpha (e_1 - e_2) + \beta (e_3 - e_4) + \gamma (e_5 - e_6) 
$$
for some $i \in \{ 1, 2, \ldots , 6 \}$ and $\alpha, \beta, \gamma \in \Z$.

Let us take $i = 1$, then 
$$
u_1^2 = e_1 + (\a + 4 \b \g) (e_1 - e_2) + 2 \g (1 + 2 \a) (e_3 - e_4) + 2 \b (1 + 2 \a) (e_5 - e_6) 
$$
and the equality $u_1^2 = u_1$ holds if and only if the following system has integer solutions:
$$
\begin{cases}
\a = \a + 4 \b \g,\\
\b = 2  \g (1 + 2 \a), \\
\g = 2 \b (1 + 2 \a). 
\end{cases}
$$ 
From the first equation follows that $\b=0$ or $\g = 0$, but from the second and the third equations we get $\b=\g = 0$. Hence,  $u_1 =  e_1 + \alpha (e_1 - e_2)$  is an idempotent in $\Z[1, 2]$.

Let us take $i = 2$, then 
$$
u_2^2 = e_2 + (\a + 4 \b \g) (e_1 - e_2) + 2 \g ( -1 + 2 \a) (e_3 - e_4) + 2 \b (-1 + 2 \a) (e_5 - e_6) 
$$
and the equality $u_2^2 = u_2$ holds if and only if the following system has integer solutions:
$$
\begin{cases}
\a = \a + 4 \b \g,\\
\b = 2  \g (-1 + 2 \a), \\
\g = 2 \b (-1 + 2 \a). 
\end{cases}
$$ 
From the first equation follows that $\b=0$ or $\g = 0$, but from the second and the third we get $\b=\g = 0$. Hence,  $u_2 =  e_2 + \alpha (e_1 - e_2)$ is an idempotent in $\Z[1, 2]$. 

Instead of considering all other cases where $i \in \{ 3, 4, 5, 6 \}$, we can note that the permutation of elements of $X$ given by
$$
\varphi = (1 3 5) (2 4 6)
$$
is an automorphism of $X$, and it induces an automorphism of  $\Z[X]$. Hence the elements $\varphi(u_1)$, $\varphi(u_2)$ have the form of $u_3$,  $u_4$ respectively, and the elements $\varphi^2(u_1)$, $\varphi^2(u_2)$ have the form of $u_5$,  $u_6$ respectively. From this observation follows that any idempotent of $\Z[X]$ lies in one of the following rings: $\Z[\{1, 2\}]$,  $\Z[\{3, 4\}]$ and  $\Z[\{5, 6\}]$. 
\end{proof}


Let us consider  automorphisms of $\Z[X]$. It is known (see page 24 of  \cite{Maq}) that the  automorphism group $\Aut(X)$ of the  quandle $X$ is the symmetric group on four letters: $\Aut(X)\cong S_4$. Thus the automorphism group  $\Aut (\mathbb{Z}[X])$ of the quandle ring $\mathbb{Z}[X]$ should contain $S_4$ as a subgroup. 
On the other hand, $\Z[X]$ contains 3 quandle rings of 2-element trivial quandles: $\Z[\{1, 2\}]$,  $\Z[\{3, 4\}]$ or  $\Z[\{5, 6\}]$.

 Let us find the group of automorphisms of $\Z[T_2]$, there $T_2 = \{ t_1, t_2 \}$ is the 2-element trivial quandle. Any endomorphism $\varphi$ of $\Z[T_2]$ sends idempotent to idempotent and, hence, has a form
$$
\varphi (t_1) = \alpha t_1 + (1 - \alpha) t_2,~~\varphi (t_2) = \beta t_1 + (1 - \beta) t_2,~~\alpha, \beta \in \Z.
$$
This endomorphism is a bijection if and only if $\beta = \alpha + \varepsilon$, where $\varepsilon \in \{ \pm 1 \}$. Hence, we have the following 

\begin{lemma} \label{ATQ}
Any automorphism $\varphi$ of $\Z[T_2]$ has a form
$$
\varphi (t_1) = \alpha t_1 + (1 - \alpha) t_2,~~\varphi (t_2) = (\alpha + \varepsilon) t_1 + (1 - \alpha  -  \varepsilon) t_2,~~\alpha  \in \Z,~~\varepsilon \in \{ \pm 1 \}.
$$
\end{lemma} 

 Let $\psi$ be a ring automorphism of $\mathbb{Z}[X]$.  Then for any $1\leq i \leq 6$, the image $\psi(e_i)$ is idempotent and thus it is one of the 3 idempotents  $\alpha e_1 +(1-\alpha)e_2$, $\beta e_3+(1-\beta)e_4$ and $\gamma e_5 +(1-\gamma)e_6$, where $\alpha, \beta, \gamma \in \mathbb{Z}$.  Thus the map $\psi$ will permute the three 2-dimensional  subspaces $\langle e_1,e_2 \rangle$, $\langle e_3,e_4 \rangle$ and $\langle e_5,e_6 \rangle$ of $\mathbb{Z}[X]$.  In the matrix form $\psi$ will be one of the following $6$ block-matrices (corresponding to the $6$ permutations of the symmetric group $S_3$):
\begin{equation*}
\begin{pmatrix}
\begin{matrix} 
A & 0 & 0 \\
0 & B& 0\\
0 & 0 &C
\end{matrix}
\end{pmatrix},
\begin{pmatrix}
\begin{matrix} 
A & 0 & 0 \\
0 & 0& B\\
0 & C&0
\end{matrix}
\end{pmatrix},
\begin{pmatrix}
\begin{matrix} 
0 & 0 & C \\
0& B & 0\\
A & 0&0
\end{matrix}
\end{pmatrix},
\begin{pmatrix}
\begin{matrix} 
0& B & 0 \\
A &0 & 0\\
0 & 0 & C
\end{matrix}
\end{pmatrix},
\begin{pmatrix}
\begin{matrix} 
0 & B & 0 \\
0 & 0& C\\
A & 0&0
\end{matrix}
\end{pmatrix},
\begin{pmatrix}
\begin{matrix} 
0 & 0& A \\
B & 0 & 0\\
0 & C&0
\end{matrix}
\end{pmatrix},
\end{equation*}
 where $A,B,C \in GL(2, \mathbb{Z})$ correspond to automorphisms of $\Z[\{1, 2\}]$,  $\Z[\{3, 4\}]$ and  $\Z[\{5, 6\}]$.

Here is an explicit example of an automorphism: 
	  \begin{eqnarray}
	  \psi(e_1)&=&e_2, \; \psi(e_2)=-e_1+2e_2, \; \psi(e_3)=e_4 \nonumber\\
	  \psi(e_4)&=&e_3+2e_4, \; \psi(e_5)=2e_5-e_6, \; \psi(e_6)=3e_5-2e_6, \nonumber
	  \end{eqnarray}
which in matrix notation has the form
	  \begin{equation*}
\psi=\begin{pmatrix}
\begin{matrix} 
0 & -1&0 & 0 &0&0 \\
1&2& 0 & 0 &0&0 \\
0 & 0 &0 & 1 &0 & 0\\
0&0&1&2&0&0\\
0 & 0 &0&0 &2&3\\
0 & 0 &0&0 & -1 &-2
\end{matrix}
\end{pmatrix}
\end{equation*}

From this observation and Lemma \ref{ATQ} follows the following

\begin{pro}\label{AutX}
Let $X_1 = \{ 1, 2 \}$, $X_2 = \{ 3, 4 \}$,  $X_3 = \{ 5, 6\}$ be the trivial subquandles of $X$. Then any automorphism $\psi$ of $\Z[X]$ can be presented in the form $\psi = \varphi_1 \varphi_2 \varphi_3 \pi$, where $\varphi_i \in \Aut(\Z[X_i])$ and $\pi$ is a permutation of the quandles $X_1, X_2, X_3$.
\end{pro}


\section{Questions for Further Research}\label{Questions:}

\begin{qst}
\begin{enumerate}
\item
 Is it true that any finite commutative quandle has a form
$$
C_{2n_1+1} \oplus C_{2n_2+1} \oplus  \ldots  \oplus C_{2n_l+1} 
$$
for some $n_i$ and $l$?

\item
 What can we say about commutative quandles in which for any elements $x, y$ the equality $x \bar{*} y = y \bar{*}  x$  holds?
 \end{enumerate}
\end{qst}

 \begin{qst}
What can we say about idempotents of integer quandle ring of  quandle of cyclic type?
  \end{qst}
  
  \begin{qst}
 Is it true that the free commutative 2-generated quandle is isomorphic to the quandle on the set
$$
\left\{ \frac{m}{2^k} ~|~m \in \mathbb{Z},  k \in \{ 0, 1, 2, \ldots \}   \right\}
$$
with operation $a * b = \frac{1}{2} (a + b)$ and generators $0$ and $1$?
\end{qst}

\section*{Competing Interests}
The authors declare that they have no competing interests.

\section*{Funding}
Valeriy G. Bardakov is supported by the Ministry of Science and Higher Education of Russia (agreement No. 075-02-2026-1339). 	Mohamed Elhamdadi was partially supported by Simons Foundation collaboration grant 712462.  
The authors thank Luc Ta, Alex Iskra and the referee for the useful  comments and suggestions.

\section*{Data Availability}
No datasets were generated or analyzed during the current study.


\end{document}